\documentclass{article}
\usepackage{amsrefs,amssymb,amsmath,amsthm,amsfonts,epsfig,graphicx}
\usepackage[utf8]{inputenc}
\usepackage[T1]{fontenc}
\usepackage[english]{babel}
\usepackage{hyperref}
\hypersetup{
    colorlinks,
    citecolor=black,
    filecolor=black,
    linkcolor=black,
    urlcolor=black
}

\theoremstyle{plain}
\newtheorem{teo}{Theorem}[section]

\newtheorem{lem}[teo]{Lemma}
\newtheorem{prop}[teo]{Proposition}		
\newtheorem{rmk}[teo]{Remark}

\theoremstyle{definition}

\DeclareMathOperator{\diam}{diam}

\DeclareMathOperator{\dist}{dist}

\newcommand{\expc}{\eta}

\newcommand{\R}    {\mathbb R}

\newcommand{\Z}  {\mathbb Z}

\renewcommand{\epsilon}{\varepsilon}

\author{Alfonso Artigue}

\title{Anomalous cw-expansive surface homeomorphisms}

\begin{document}
\date{\today}
\maketitle

\begin{abstract}
We prove that the genus two surface admits a cw-expansive homeomorphism 
with a fixed point whose local stable set is not locally connected.
\end{abstract}

\section{Introduction}

In \cites{L,Hi} Lewowicz and Hiraide proved that every expansive homeomorphism of a compact surface $S$ is 
conjugate with a pseudo-Anosov diffeomorphism. 
Recall that a homeomorphism $f\colon S\to S$ is \emph{expansive} if there is 
$\expc>0$ such that if $\dist(f^n(x),f^n(y))\leq\expc$ for all $n\in\Z$ then 
$x=y$. In \cite{Ka93} Kato introduced a generalization of expansivity called \emph{continuum-wise expansivity}. 
We say that $f$ is \emph{cw-expansive} if there is $\expc>0$ such that if $C\subset S$ is a continuum (compact connected) 
and $\diam(f^n(C))\leq\expc$ for all $n\in\Z$ then $C$ is a singleton. 
In the works of Kato on cw-expansivity we find several generalizations of results holding for expansive homeomorphisms.
Also, he found new phenomena as for example a cw-expansive homeomorphism with infinite topological entropy.
In this paper we investigate the possibility of extending results from \cites{L,Hi} for a cw-expansive surface 
homeomorphism.

A key concept in dynamical systems is that of the stable set of a point. 
Given a homeomorphism $f\colon S\to S$ and $\epsilon>0$ we define the $\epsilon$-\emph{stable set} of a point $x\in S$ 
as
\[
  W^s_\epsilon(x)=\{y\in S:\dist(f^n(x),f^n(y))\leq\epsilon \hbox{ for all }n\geq 0\}.
\]
For a hyperbolic set it is well known that
local stable sets are embedded submanifolds (the invariant manifold theorem). 
In the papers \cites{L,Hi} they prove that if $f\colon S\to S$ is expansive then 
the connected component of $x$ in $W^s_\epsilon(x)$ is a locally connected set. 
This implies the arc-connection of this components and allows them to prove that 
each local stable set is a finite union of arcs. 
In some sense it is an invariant manifold theorem for expansive homeomorphisms of surfaces.
After this, they prove the conjugacy with a pseudo-Anosov diffeomorphism, 
giving a complete classification of such dynamics. 

Some cw-expansive homeomorphisms of surfaces are not expansive, see \cites{ArNexp, APV, PPV, PaVi}.
In these examples the components of local stable sets are locally connected. 
The purpose of this paper is to construct a cw-expansive homeomorphism of a compact surface 
with a point whose local stable set is connected but it is not locally connected.

\section{The example}
\label{secAno}
The example is a variation of those in \cites{ArNexp, APV}. 
We start defining a homeomorphism of $\R^2$ such $(0,0)$ as a fixed point and 
its stable set is not locally connected. 
Then, this anomalous saddle is \emph{inserted} in a derived from Anosov diffeomorphism of the torus. 
Finally, this anomalous derived from anosov system is connected via a wandering tube with a usual derived from Anosov to obtain our example.

\subsection{An anomalous saddle point}
\label{secIrregular}
First, we will construct a plane homeomorphism $f$ with a fixed point at the origin 
whose local stable set is connected but 
not locally connected.
The homeomorphism will be defined as the composition of a piece-wise linear transformation $T$ and a time-one 
map of a flow $\phi$. This flow will have a non-locally connected set $E$ of fixed points.

We start with the linear part of the construction.
Let $T_i\colon \R^2\to\R^2$, for $i=1,2,3$, be the linear transformations defined by 
 $T_1(x,y)=(\frac x2,\frac y2)$, $T_2(x,y)=(\frac x2,2y)$, $T_3(1,1)=(\frac12,\frac12)$, $T_3(0,1)=(0,2)$.
Define the piece-wise linear transformation $T\colon \R^2\to\R^2$ as 
\[
 T(x,y)=\left\{ 
 \begin{array}{l}
    T_1(x,y)\hbox{ if } x\geq y\geq 0,\\
    T_2(x,y)\hbox{ if } x\leq 0 \hbox{ or } y\leq 0,\\
    T_3(x,y)\hbox{ if } y\geq x\geq 0.
 \end{array}
 \right.
\]
In Figure \ref{sillaLoca} we illustrate the definition of $T$.
\begin{figure}[ht]
\center{\includegraphics[scale=.7]{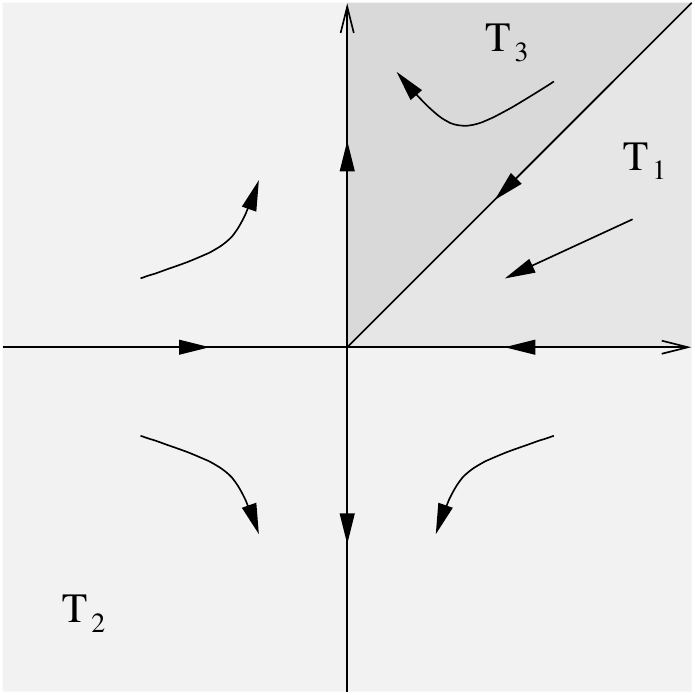}}
\caption{The action of the piece-wise linear transformation $T$.}
\label{sillaLoca}
\end{figure}

Now we define the non-locally connected plane continuum $E$.
Some care is needed in order to be able of relate this set with the transformation $T$.
Define the sets:
\[
\begin{array}{l}
C(a)=\{(a,y)\in\R^2:0\leq y\leq a\}\hbox{ for }a>0, \\
D_1=\cup_{i=1}^\infty C(\frac12+\frac1{2^i}), \\
D_{n+1}=T_1(D_n)\hbox{ for all }n\geq 1, \\
D=\cup_{n\geq 1} D_n.
\end{array}
\]
Also consider the non-locally connected continuum $E=D\cup([0,1]\times\{0\})$ shown in Figure \ref{conjuntoE}. 
\begin{figure}[ht]
\center{\includegraphics[scale=.8]{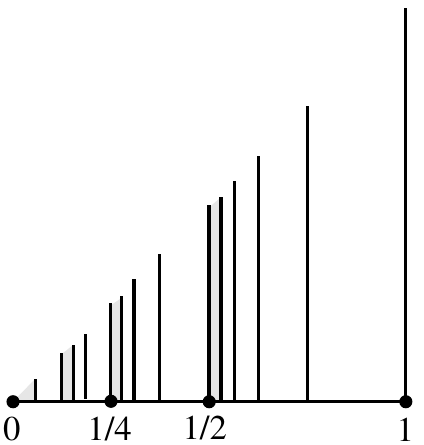}}
\caption{The non-locally connected set $E$.}
\label{conjuntoE}
\end{figure}

Now we will define a flow related with the set $E$.
Consider the continuous function $\rho\colon \R^2\to\R$ defined by 
$$\rho(p)=\dist(p,E)=\min\{\dist(p,q):q\in E\}$$
and the vertical vector field $X\colon\R^2\to\R^2$ defined as 
$$X(p)=(0,\rho(p)).$$ 
Since 
\[
 |\dist(p,E)-\dist(q,E)|\leq\dist(p,q)
\]
for all $p,q\in\R^2$, we have that $\rho$ is Lipschitz.
Therefore, by Picard's theorem, $X$ has unique solutions and we can  
consider the flow $\phi\colon\R\times\R^2\to\R^2$ induced by $X$. 
Since $\|X(p)\|\leq \|p\|$ for all $p\in \R^2$ we have that every solution is defined for all $t\in\R$.

Let $f\colon\R^2\to\R^2$ be the homeomorphism 
$$f=\phi_1\circ T,$$ 
where $\phi_1\colon\R^2\to\R^2$ is the time-one homeomorphism associated to the vector field $X$.

\begin{prop}
 The homeomorphism $f$ preserves the vertical foliation on $\R^2$.
\end{prop}

\begin{proof}
 It follows because $\phi_t$ and $T$ preserves the vertical foliation.
\end{proof}

Consider the region 
\begin{equation}
\label{eqR1}
R_1=\{(x,y)\in[0,1]\times[0,1]:x\geq y\}.  
\end{equation}

\begin{lem}
\label{lemaR1}
 For all $p\in R_1$ it holds that $\rho(T(p))=\frac 12 \rho(p)$ and 
 $$\phi_t(T(p))=T(\phi_t(p))$$ 
 if $\phi_t(p)\in R_1$ and $t\geq 0$.
\end{lem}

\begin{proof}
 By the definition of $T$ we have that $T(p)=T_1(p)=\frac12p$ for all $p\in R_1$. 
 Given $p\in R_1$ consider $q\in E$ such that $\rho(p)=\dist(p,q)$. 
 Then $\rho(T(p))=\dist(T(p),T(q))$ and $\rho(T(p))=\frac 12 \rho(p)$.
  
 Consider $t\geq 0$ such that $\phi_t(p)\in R_1$. Since $X$ is a vertical vector field we have that $\phi_{[0,t]}(p)\subset R_1$. 
 For $s\in(0,t)$, if $q=\phi_s(p)$ then 
 \[
  X(T(q))=(0,\rho(T(q)))=\left(0,\frac12\rho(q)\right)=d_qT(X(q)).
 \]
 Therefore, $\phi_s(T(p))=T(\phi_s(p))$ for $s\in(0,t)$ and consequently for $s=t$.
\end{proof}

Define the stable set of the origin as usual by $$W^s_f(0)=\{p\in\R^2:\lim_{n\to+\infty}\|f^n(p)\|=0\}.$$

\begin{prop}
 For the homeomorphism $f\colon\R^2\to\R^2$ defined above it holds that $$W^s_f(0)\cap([0,1]\times[0,1])=E.$$
\end{prop}

\begin{proof} 
First notice that $E\subset W^s_f(0)$ because for all $p\in E$ and $t\in\R$ we have that 
$\phi_t(p)=p$ and $T(p)=\frac12p$. Then $f(p)=\frac12p$ for all $p\in E$. 

Now take a point $p\in [0,1]\times[0,1]$. 
For $p\notin R_1$, the set defined in (\ref{eqR1}), it is easy to see that $f^n(p)\to\infty$ as $n\to+\infty$.
Assume that $p\in R_1\setminus E$. 
We will show that $p\notin W^s_f(0)$. 
It is sufficient to show that for some $n>0$ the point $f^n(p)$ is not in $R_1$.
By contradiction, assume that $f^n(p)\in R_1$ for all $n\geq 0$. 
Then, by Lemma \ref{lemaR1} we know that 
$$f^n(p)=(\phi_1\circ T)^n(p)=T^n(\phi_n(p)).$$
Notice that $\phi_1^n=\phi_n$. 
Then, it only rests to prove that $\phi_n(p)\notin R_1$ for some $n>0$. 
But this is easy because the velocity of $\phi_t(p)$ is $\rho(\phi_t(p))$ and 
this velocity increases with $t$. 
\end{proof}

\subsection{A variation of a derived from Anosov}
\label{AnomCwexp}

We start recalling some properties of what is known as 
a derived from Anosov diffeomorphisms. 
The interested reader should consult \cite{Robinson}*{Section 8.8} for a construction of such a map 
and detailed proofs of its properties. 
A derived from Anosov is a $C^\infty$ diffeomorphism $f\colon T^2\to T^2$ 
of the two-dimensional torus such that: it satisfies Smale's axiom A and its non-wandering set consists 
of an expanding attractor and a repeller fixed point $p\in T^2$. 
The expanding attractor is locally a Cantor set times an arc, and it has two hyperbolic 
fixed points of saddle type $q$ and $q'$ as in Figure \ref{figDA}.

\begin{figure}[ht]
\begin{center}  
  \includegraphics[scale=.7]{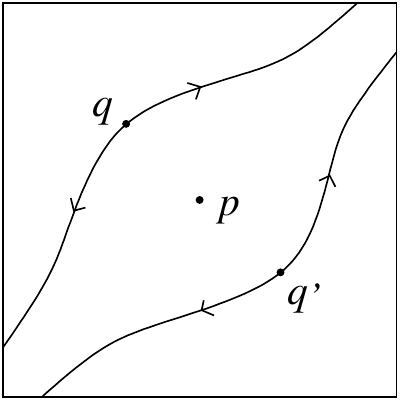}
  \caption{The derived from Anosov diffeomorphism on the two-dimensional torus.}
  \label{figDA}
\end{center}
\end{figure}

We will assume that there is a local chart $\varphi\colon D\to T^2$, 
defined on the disc $D=\{x\in\R^2:\|x\|\leq 2\}$,
such that
\begin{enumerate}
\item $\varphi(0)=p$,
\item the pull-back of the stable foliation by $\varphi$
is the vertical foliation on $D$ and
\item $\varphi^{-1}\circ f\circ \varphi(x)=4x$
for all $x\in D$ with $\|x\|\leq 1/2$.
\end{enumerate}

Now we will \emph{insert} the anomalous saddle in the derived from Anosov.
Let $q$ be the hyperbolic fixed point shown in Figure \ref{figDA}. 
Consider a topological rectangle $R_q$ covering a half-neighborhood of $q$ as in Figure \ref{figDARect}.
\begin{figure}[ht]
\begin{center}  
  \includegraphics{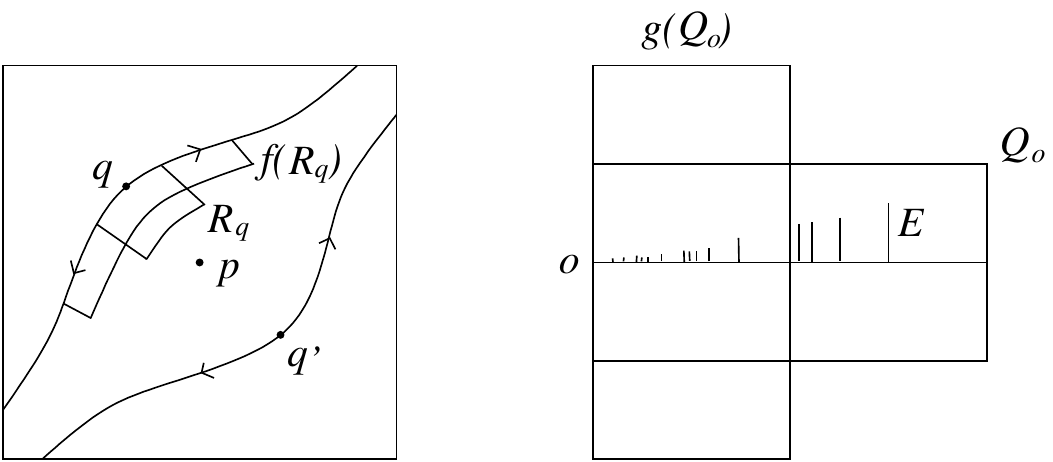}
  \caption{Topological rectangles on the derived from Anosov (left) and on the anomalous saddle (right).}
  \label{figDARect}
\end{center}
\end{figure}
Consider the homeomorphism with an anomalous saddle fixed point defined in Section \ref{secIrregular}. 
Call this homeomorphism $g$ (to avoid confusion with the derived from Anosov $f$). 
Denote by $o$ its fixed point (the origin of $\R^2$) and 
take a rectangle $Q_o\subset \R^2$, similar to $R_p$, as in Figure \ref{figDARect}.
Now we can \emph{replace} $R_q$ with $Q_o$ and define what we call a 
\emph{derived from Anosov with an anomalous saddle} as in Figure \ref{figAnoDA}. 
\begin{figure}[ht]
\begin{center}  
  \includegraphics{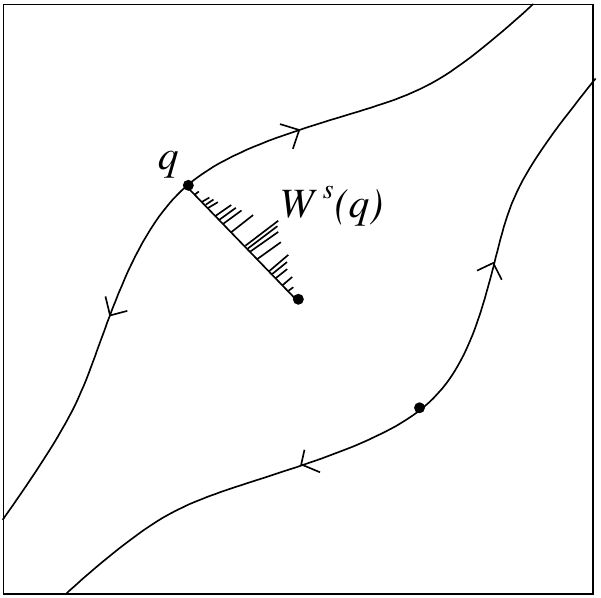}
  \caption{Derived from Anosov with an anomalous saddle fixed point $q$.}
  \label{figAnoDA}
\end{center}
\end{figure}

\subsection{Anomalous cw-expansive surface homeomorphism}

In this section we finish the construction with ideas from \cites{ArNexp,APV}.
Consider $S_1$ and $S_2$ two disjoint copies of the torus $\R^2/\Z^2$.
Let $f_i\colon S_i\to S_i$, $i=1,2$, be two homeomorphisms such that:
\begin{itemize}
\item $f_1$ is the derived from Anosov with an anomalous saddle from the previous section, denote by  
$p_1\in S_1$ the source fixed point of $f_1$,
\item $f_2$ is the inverse of the derived from Anosov (the usual one) with a sink fixed point at $p_2\in S_2$.
\end{itemize}
Consider local charts $\varphi_i\colon D_2\to S_i$, $i=1,2$, where $D_2$ is the compact 
disk
$$D_2=\{x\in\R^2:\|x\|\leq 2\},$$
such that:
\begin{enumerate}
 \item $\varphi_i(0)=p_i$,
\item the pull-back of the unstable foliation by $\varphi_2$
is the vertical foliation on $D_2$ and
\item $\varphi_1^{-1}\circ f^{-1}_1\circ \varphi_1(x)=\varphi_2^{-1}\circ f_2\circ \varphi_2(x)=x/4$
for all $x\in D$.
\end{enumerate}

Consider the open disk $$D_{1/2}=\{x\in\R^2:\|x\|<1/2\}$$ 
and the compact annulus $$A=D_2\setminus D_{1/2}.$$
Define $\psi\colon A\to A$ as the inversion $\psi(x)=x/\|x\|^2$.
The pull-back of the unstable foliation on $S_2$ by $\varphi_2\circ\psi$ on the annulus $A$ 
is shown in Figure \ref{figFols}.
\begin{figure}[ht]
\begin{center}  
  \includegraphics{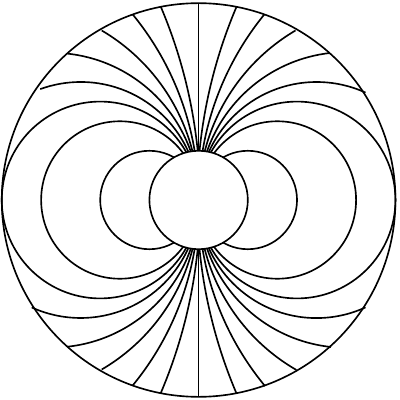}
  \caption{Unstable foliation of $f_2$ on the annulus, in the local chart $\varphi_2\circ\psi$.}
  \label{figFols}
\end{center}
\end{figure}
On the disjoint union $S_3=[S_1 \setminus \varphi_1(D_{1/2})]\cup [S_2\setminus \varphi_2(D_{1/2})]$
consider the equivalence relation generated by
\[
 \varphi_1(x)\sim \varphi_2\circ\psi (x)
\]
for all $x\in A$. Denote by $[x]$ the equivalence class of $x$.
The surface $S=S_3/\sim$ is the genus two surface if equipped with the quotient topology.
Consider the homeomorphism $f\colon S\to S$ defined by
\[
 f([x])=\left\{
\begin{array}{ll}
\left[f_1(x)\right] &\hbox{ if } x\in S_1 \setminus \varphi_1(D_{1/2})\\
\left[f_2(x)\right] &\hbox{ if } x\in S_2 \setminus \varphi_2(D_2)\\
\end{array}
\right.
\]

For $x\in S$ and $\expc>0$ define the set
\[
  \Gamma_\expc(x)=W^s_\expc(x)\cap W^u_\expc(x).
\]
\begin{rmk}
  In order to prove that a homeomorphism $f$ is cw-expansive it is equivalent to find $\expc>0$ such that 
  $\Gamma_\expc(x)$ is totally disconnected for all $x\in S$.
\end{rmk}

\begin{teo}
\label{mainteo}
There are cw-expansive homeomorphisms of the genus two surface having a 
fixed point whose local stable set is connected but it is not 
locally connected.
\end{teo}

\begin{proof}
Define $A_S=[\phi_1(A)]$ the annulus on $S$ corresponding to $A$. 
We will perturb the homeomorphism $f$ defined above on the annulus $A_S$. 
First note that the non-wandering set of $f$ is expansive and dynamically isolated, 
i.e. there is a neighborhood $U$ of the non-wandering set $\Omega$ such that if $f^n(x)\in U$ for all 
$n\in\Z$ then $x\in\Omega$. 
Also note that for every wandering point $x\in S$ there is $n\in\Z$ such that 
$f^n(x)\in A_S$. 
Therefore, it is sufficient to prove that there is a homeomorphism $g\colon S\to S$ such that 
$f|A_S=g|A_S$ and 
there is $\delta>0$ such that for each $x\in A_S$ the intersection 
$\Gamma_\expc(x)$ is totally disconnected. 
In Figure \ref{figFols} we have the picture of the unstable foliation on $A_S$ (or in local charts). 
The problem is that the stable sets do not make a foliation, this is because there is an anomalous saddle. 
Then, it is convenient to consider the stable partition, i.e., the partition defined by 
the equivalence relation of being positively asymptotic. 
This partition is illustrated in Figure \ref{figStPart}.
\begin{figure}[ht]
\begin{center}  
  \includegraphics{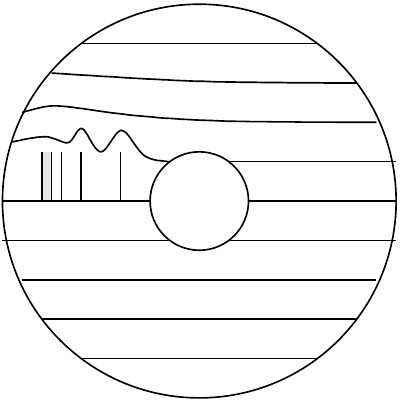}
  \caption{Stable partition on the annulus $A_S$.}
  \label{figStPart}
\end{center}
\end{figure}
We know that the unstable leaves are circle arcs, as in Figure \ref{figFols}. 
Therefore, it is sufficient to consider a $C^0$ perturbation $g$ of $f$ supported on $A_S$, 
such that the stable partition of $g$ in the annulus contains no circle arc, in local charts. 
See comments below.
By the previous comments this implies that $g$ is cw-expansive.
Since $g$ coincides with $f$ outside $A_S$, we have that $g$ has an anomalous saddle with 
non-locally connected stable set. This finishes the proof.
\end{proof}

The example has further properties that we wish to remark. 
Given $N\geq 1$ we say that $f$ is $N$-\emph{expansive} \cite{Mo12} if there is $\expc>0$ such that 
$|\Gamma_\expc(x)|\leq N$ for all $x\in S$, where $|A|$ stands for the cardinality of the set $A$.
We have that the example of the previous proof is not $N$-expansive for all $N\geq 1$ 
because there are point with $|\Gamma_\expc(x)|=\infty$ for arbitrarilly small values of $\expc$.

We say that a probability measure $\mu$ on $S$ is an \emph{expansive measure} \cite{MoSi}
if there is $\expc>0$ such that $\mu(\Gamma_\expc(x))=0$ for all $x\in S$. 
Obviously, if $\mu$ is an expansive measure then $\mu(x)=0$ for all $x\in S$, i.e $\mu$ is non-atomic.
In \cite{AD} it is shown that every non-atomic probability measure is expansive 
if and only if there is $\expc>0$ such that $|\Gamma_\expc(x)|\leq|\Z|$ for all $x\in S$.
This property is called \emph{countable-expansivity} and our example satisfies this condition. 

In the generalized pseudo-Anosov shown in \cites{PPV,PaVi} there is a finite number of \emph{spines} (or 1-\emph{prongs}), 
i.e. points whose local stable sets do not separate arbitrarilly small neighborhoods. 
This is a cw-expansive homeomorphism on the two-sphere. 
Our example has a countable set of spines, namely, 
the points in the set $E$ of Figure \ref{conjuntoE} in the line $y=x$ give rise to spines in the example.
As explained in \cite{PPV} the generalized pseudo-Anosov of the two-sphere has points with its local stable set 
non-locally connected. But the components are arcs. Our example has connected components not being locally connected. 
It seems to be the case that if we start with a set like the graph of $\sin(1/x)$, in place of the set $E$, 
we can obtain an anomalous saddle with no arc-connected stable set. Notice that the set $E$ is arc-connected.

Let us finally give some questions. 
May an example as in Theorem \ref{mainteo} be smooth? 
Can it be transitive, i.e. to have a dense orbit?

\noindent Departamento de Matemática y Estadística del Litoral, Salto-Uruguay\\
Universidad de la República\\
E-mail: artigue@unorte.edu.uy
\end{document}